\documentclass[10pt]{amsart}
\usepackage{graphicx}
 \usepackage[parfill]{parskip}
 \usepackage{amsmath}
\usepackage{amsthm}
\usepackage{amssymb}
\usepackage[T1]{fontenc}
 \newtheorem{theorem}{Theorem}
 
\newtheorem{lemma}[theorem]{Lemma}

\newtheorem{corollary}[theorem]{Corollary}

\def\a{\alpha}

\def\E{\mathsf {E}}

\def\d{\delta}

\def\F{{\mathbb F}}

\def\P{{\mathcal P}}

\def\u{\underline}
\def\supp{\mathrm{supp\, }}

\title{On the restriction problem for discrete paraboloid in lower dimension}

\author{Misha Rudnev}
\address{Misha Rudnev, Department of Mathematics, University of Bristol,
  Bristol BS8 1TW, United Kingdom}
\email{m.rudnev@bristol.ac.uk}

\author{Ilya D. Shkredov}
\address{Ilya D. Shkredov, Steklov Mathematical Institute, Division of Number Theory, ul. Gubkina, 8, Moscow, Russia, 119991 and IITP RAS, Bolshoy Karetny per. 19, Moscow, Russia, 127994 and  MIPT, Institutskii per. 9, Dolgoprudnii, Russia, 141701}
\email{ilya.shkredov@gmail.com}

\begin{document}
\begin{abstract} 
	We apply geometric incidence estimates in positive characteristic to prove the optimal $L^2 \to L^3$ Fourier extension estimate for the paraboloid in the four-dimensional vector space over a prime residue field. In three dimensions, when $-1$ is not a square,  we prove an $L^2 \to L^{\frac{32}{9} }$ extension estimate, improving the previously known exponent $\frac{68}{19}.$ \end{abstract}

\maketitle

\section{Introduction}
Let $F=\F_p$, the prime residue field of characteristic $p>2$, where $p$ is viewed as asymptotic parameter, consider the vector space $V=F^d$, $d\ge 3$, with the standard scalar product.  Define the paraboloid $\P$  as 
$$
	\mathcal{P} = \{ (\u x=(x_1,\ldots, x_{d-1}),\, \u x\cdot \u x) ~:~ \u x\in F^{d-1} \} = \{ (x_1,\ldots x_{d-1}, x^2_1 + \ldots +  x^2_{d-1}) \} \subset V \,.
$$
The Fourier extension problem is bounding some Lebesgue norm on $V$ of the inverse Fourier transform of a complex-valued function $f$ on  $\mathcal{P} \subset V^*$, the paraboloid in the Fourier space (or more generally some degree $2$ or higher irreducible codimension-one variety)  in terms of some Lebesgue norm of the function $f$ on $\mathcal{P}$. An equivalent question, with dual Lebesgue exponents, is bounding the norm of the restriction to $\mathcal P$ of the Fourier transform of a function $g$ on $V$ in terms of the norm of $g$.

We consider the specific case where one of the Lebesgue norms is $L^2$, which can be reduced nicely to geometric incidence combinatorics, and dimensions three and four. In $d=4$ we prove the optimal $L^2 \to L^3$ extension estimate. Optimality, in view of normalisation described below, is readily verified by calculating the inverse Fourier transform of the characteristic function of the lift of an  isotropic line in the base space $F^3$ on $\mathcal{P}$. Indeed,  the paraboloid $\mathcal P$  contains, in particular, the {\em null sphere} -- alias {\em isotropic cone} -- that is the set $S_0\times\{0\}\subset F^4$, where 
$$S_0 = \{\u x \in F^3:\, \u x\cdot \u x =0\}.$$ Clearly, $S_0$ is a cone, for if $\u x\in S_0$, then so is its multiple by a scalar. Such a three-vector $\u x$ is called {\em isotropic}, as well as any line, whose direction vector is isotropic. By non-degeneracy of the dot product, if $\u x,\,\u y$ are nonzero isotropic vectors in $F^3$, with $\u x\cdot \u y=0$, then one is a scalar multiple of the other.

In $d=3$ two isotropic directions in the base space $F^2$ exist if $-1$ is a square in $F$, that is $p\equiv 1 \pmod 4$ and do not exist if $p\equiv  3 \pmod 4$. In the former case, by taking the inverse Fourier transform of an isotropic line, one concludes that the best possible extension Fourier exponent from $L^2 $ is $4,$ which is also the Stein-Tomas exponent (which means that it follows from the basic Fourier estimates \eqref{f:g_hat_2} and \eqref{f:g_hat_3} alone). However, if $-1$ is not a square, the best possible  extension exponent from $L^2 $ is an open question. Mockenhaupt and and Tao \cite{MoT} proved a threshold  $L^2 \to L^{\frac{18}{5}}$  estimate (but for the endpoint) and conjectured that the best exponent should be $3$. Their exponent $\frac{18}{5}$ was improved to $\frac{18}{5}-\delta$, for  $0<\delta<\frac{1}{1035}$ by Lewko \cite{Lew}. Stevens and de Zeeuw \cite{SdZ} remark that their new incidence theorem -- presented here as Lemma \ref{lem:2d} -- would justify $0<\delta<\frac{2}{95}$ in  Lewko's argument, thus claiming the exponent $\frac{68}{19}$. Here we push it a little further, namely  $0<\delta\leq \frac{2}{45}$, proving an
$L^2 \to L^{\frac{32}{9}}$ extension estimate, owing to a more meticulous application of the Stevens-de Zeeuw theorem. 

The restriction problem over  the real field has a reputed history, which we do not aim to present; since the 2000s, after having been set up by Mockenhaupt and Tao \cite{MoT}, the question has also been studied in the finite field setting. Introduction-level discussion of the discrete paraboloid case can be found in Green's lecture notes \cite[Sections 6--9]{Green}. Heuristically, restriction phenomena are closely linked to geometric incidence laws, governing intersections of lower-dimensional affine subspaces in a vector space.  Recent  progress is due to  Lewko \cite{Lew}, \cite{Lewko}, as well as Iosevich, Koh and Lewko \cite{IKL}. These works -- see especially \cite{Lew}, \cite{Lewko} -- also introduce the subject at length and breadth and contain exhaustive reference lists.  Remarkably, in \cite{IKL}  optimal extension/restriction estimates for the paraboloid in even dimensions $\geq 6$ (but for the endpoint  exponent in dimension $6$) were established by using a somewhat crude geometric incidence machinery based entirely on discrete Fourier analysis, involving  Gauss and Kloosterman sum estimates. By somewhat crude we mean, heuristically, that this machinery works well, provided that the sets of geometric objects involved are sufficiently large in terms of the cardinality of the finite field. This approach  in \cite{IKL} turns out to be powerful enough to have led to optimal estimates in higher dimension, however in dimension $4$ it only allows for a partial result.

In this note, if the dimension $d=4$, we use sharper incidences results from the first author's paper \cite{Rud}, which turn out to be ideally suited to study the restriction problem on the paraboloid in this dimension. This allows us to settle the question. We must admit in comparison that the methods of \cite{IKL}-\cite{Lewko} work in any finite field $\F_q$ of odd characteristic, while we are forced to confine ourselves to $F=\F_p$. The obstacle is that the characteristic $p$ does appear in the statement of the point-plane incidence theorem, Lemma \ref{lem:Misha+}, that we fetch from the first author's work \cite{Rud} -- see the latter paper for discussion why replacing (under some constraint) $p$ by $q$ is likely to be a difficult structural problem.

The argument in dimension $3$ follows roughly the same lines if instead of  Rudnev's point-plane incidences theorem one uses the point-line Szemer\'edi-Trotter type incidence theorems in $F^2$ due to Stevens and de Zeeuw \cite{SdZ}. The fact that the theorem is useful for the restriction problem  was observed in the latter paper, claiming the extension exponent $\frac{68}{19}$. Here we develop a more elaborate way of applying the latter point-line incidence bound along the lines of the well-known, and sharp, application of the Szemer\'edi-Trotter theorem to the repeated angle in the plane bound, developed by Pach and Sharir \cite{PS}.

We remark that the Stevens-de Zeeuw theorem is in some sense a corollary of Rudnev's theorem, which in turn can be viewed as an adaptation of the breakthrough development of the polynomial method by Guth and Katz \cite{GK} (at this point we stop the genealogical detour) which has also inspired progress in the restriction problem over the reals, due to Guth \cite{Guth1}, \cite{Guth2}. 

As far as the notation is concerned, we follow \cite{IKL}, using the counting norm on $V$, thus defining the Fourier transform of a function $g:V\to \mathbb C$ as
$$
\hat g(\xi) := \sum_{x\in V} g(x) e(-x\cdot \xi)\,, 
$$
where $e$ is a non-trivial additive character. In the inverse Fourier transform summation is replaced by taking the expectation. As in the real prototype of the question, for a function $f$ on $\mathcal P,$ its inverse Fourier transform is denoted/defined as
$$
(fd\sigma)^\lor(x) := \frac{1}{p^{d-1}} \sum_{\xi\in \mathcal P} f(\xi) e(x\cdot\xi)\,,
$$
the notation $d\sigma$ standing for the normalised ``surface area'' on $\mathcal P$, assigning to each point on $\mathcal P$ the mass $|\mathcal P|^{-1}=1/p^{d-1}$ ($|\cdot|$ standing in particular, as usual, for finite cardinality).

Thus a Lebesgue $L^q$-norm of $f$ on $\mathcal P$ is defined as
$$
\|f\|_{L^q(\mathcal P,d\sigma)}:= \left( \frac{1}{p^{d-1}} \sum_{\xi\in \mathcal P} |f(\xi) |^q \right)^{\frac{1}{q}}\,,
$$
while for a function $g$ on $V$ it is 
$$
\|g\|_{L^q(V)}:= \left( \sum_{x\in V} |g(x) |^q \right)^{\frac{1}{q}}\,.
$$
Finally, we use the symbols $\ll$ (as well as $O(\,)$) and $\gg$ in the standard way to suppress absolute constants, as well as $\sim$ for an approximate equality of two quantities within a constant factor.

Our main results are as follows.
\begin{theorem} Let $f$ be a function on $\mathcal P\subset \mathbb F_p^4$. Then 
$$
\|(fd\sigma)^\lor \|_{L^3(\F_p^4)} \ll \|f\|_{L^2(\mathcal P,d\sigma)}\,.
$$
\label{t:rest}\end{theorem}
\begin{theorem} Let $f$ be a function on $\mathcal P\subset \mathbb F_p^3$ and $p\equiv 3 \pmod 4$. Then 
$$
\|(fd\sigma)^\lor \|_{L^{\frac{32}{9}}(\F_p^3)} \ll \|f\|_{L^2(\mathcal P,d\sigma)}\,.
$$
\label{t:rest3}\end{theorem}

\section{Preliminaries} 

For a set $G\subseteq V$, with the fourth coordinate $h\in F$, define $G_h \subseteq \mathcal{P}$ as the horizontal $h$-slice of $G$, lifted to $\mathcal P$, that is
$$
G_h := \{ (\u x,\u x\cdot \u x): \,(\u x,h)\in G\}\,.
$$

For a finite set $X$ in an abelian group, we define the usual additive energy as
\begin{equation}\label{energy}
\E(X) = |\{(x,y,z,u)\in X^4:\,x+y=z+u\}|\,.\end{equation}

We begin with the key preliminary consideration, originating in \cite{MoT}, which  reduces restriction bounds to estimating additive energy of sets on $\mathcal P$. We essentially quote \cite[Lemma 2.1]{IKL}, where a thorough sketch of the proof  is given. See \cite[Corollary 25, Lemma 29]{Lewko} for more details and references as to the claim. 

\begin{lemma}
	\label{l:g^-energy}
	Let $g: V \to \mathbb{C}$ be a function such that $\| g\|_\infty \leq 1$ on its support $G$. 
	Then
	\begin{equation} \label{e:g^-energy}
	\| \hat{g} \|_{L^2 (\P,d\sigma)} \ll |G|^{\frac{1}{2}} + |G|^{\frac{3}{8}} 
	p^{-\frac{d-2}{8}}
	\left( \sum_{h\in F} \E^{\frac{1}{4}} (G_h) \right)^{\frac{1}{2}} \,.
	\end{equation}
\end{lemma}
Note that in the left-hand side of \eqref{e:g^-energy} the paraboloid $\mathcal P$ lives in the Fourier space $V^*$, while in the right-hand side 
the sets $G_h$ arise on $\mathcal P\subset V$.
This can be seen by inspection of the proof of Lemma \ref{l:g^-energy}, which uses Gauss sums and the H\"older inequality.  In addition, we will use two more basic Stein-Tomas type estimates, also obtained in this vein for $\| \hat{g} \|_{L^2 (\P,d\sigma)}$, see \cite[formulae (1.6), (1.7)]{IKL}:
\begin{equation}\label{f:g_hat_2}
	\| \hat{g} \|_{L^2 (\P,d\sigma)} \ll |G|^{\frac{1}{2}} + |G| p^{-\frac{d-1}{4}} 
\end{equation}
	and
\begin{equation}\label{f:g_hat_3}
	\| \hat{g} \|_{L^2 (\P,d\sigma)} \ll p^{\frac{1}{2}} |G|^{\frac{1}{2}} \,.
\end{equation}

\section{Proof of Theorem \ref{t:rest}}
In this this section we set $d=4$, so $V=F^4$. For a vector $x = (x_1,\dots, x_4)\in V$ we write $x=(\u x,h)$, so $\underline{x}=(x_1,x_2,x_3)$ and $h=x_4$.

We complement the bounds in the previous section by a key incidence bound in the next lemma, as follows.
\begin{lemma}
\label{l:par_energy} 
	Let $A\subseteq \mathcal{P}\subset V$. 
	Then
\begin{equation}\label{f:E_par_lemma}
\E(A) \ll \frac{|A|^3}{p} + |A|^{\frac{5}{2}} + p^2|A| \,.
\end{equation}	
\end{lemma}
The proof of Lemma \ref{l:par_energy}  is presented in a separate section. We now put together the estimates we have so far in this section and show how they imply Theorem \ref{t:rest}.


\begin{corollary}
	Let $g: V \to \mathbb{C}$ be a function such that $\| g\|_\infty \leq 1$ on its support $G$. 
	Then  
\begin{displaymath}
	\| \hat{g} \|_{L^2 (\P,d\sigma)} \ll \left\{ \begin{array}{ll}
|G|^{\frac{1}{2}} + |G| p^{-\frac{3}{4}} & \textrm{for } \;1\le |G| \le p^{\frac{9}{4}}\,, \\
 |G|^{\frac{1}{2}} p^{\frac{3}{8}} & \textrm{for } \;p^{\frac{9}{4}}\le |G| \le p^{\frac{7}{3}}\,, \\
|G|^{\frac{11}{16}} p^{-\frac{1}{16}} & \textrm{for } \;p^{\frac{7}{3}} \le |G| \le p^3\,,  \\
|G|^{\frac{1}{2}} p^{\frac{1}{2}} & \textrm{for } \;p^3 \le |G| \le p^4\,.
\end{array} \right.
\end{displaymath}
\label{cor:g_hat_bounds}
\end{corollary}
\begin{proof}
	By Lemma \ref{l:g^-energy} and Lemma \ref{l:par_energy}, we have 
$$
	\| \hat{g} \|_{L^2 (\P,d\sigma)} \ll  |G|^{\frac{1}{2}}  \; + \;|G|^{\frac{3}{8}} p^{-\frac{1}{4}} \left(  \sum_{h\in F} p^{-\frac{1}{4}}|G_h|^{\frac{3}{4}} +  |G_h|^{\frac{5}{8}} + p^{\frac{1}{2}}|G_h|^{\frac{1}{4}} \right)^{\frac{1}{2}} \,.
$$
	Hence, by the H\"older inequality, 
$$
\| \hat{g} \|_{L^2 (\P,d\sigma)} 
	\ll
		 |G|^{\frac{3}{4}} p^{-\frac{1}{4}} + |G|^{\frac{11}{16}} p^{-\frac{1}{16}}  + |G|^{\frac{1}{2}} p^{\frac{3}{8}}\,.
$$
	
	Combining this with inequalities \eqref{f:g_hat_2}, \eqref{f:g_hat_3} completes the proof of Corollary \ref{cor:g_hat_bounds}. 
\end{proof}


We are now ready to prove Theorem \ref{t:rest}. 

\begin{proof}[Proof of Theorem \ref{t:rest}]

By duality,  for all $g : V \to \mathbb{C}$, it suffices to establish the restriction estimate $\| \hat{g} \|_{L^2 (\P,d\sigma)} \ll 1$, given that 
\begin{equation}\label{f:p_restriction}
	\sum_{x\in V} |g(x)|^{\frac{3}{2}} = 1 \,.
\end{equation}

We  start out with dyadic decomposition of the set of level sets of $g$. Namely,  for  $i=0,1,\ldots,L\ll \log p,$ define $G_i$ to be the set of $x\in V$, where $g(x)\sim 2^{-i}$ up to a factor of $2$, and $g_i$ is the restriction of $g$ to $G_i$. 

Clearly, from \eqref{f:p_restriction} one has $|\supp g_i| \ll 2^{\frac{3}{2}i}$.

By the triangle inequality, followed by invoking Corollary \ref{cor:g_hat_bounds}, we obtain
\begin{equation} \begin{aligned}
	\| \hat{g} \|_{L^2 (\P,d\sigma)}  & \ll \;\;\;\;\;\;\;\sum_{i=0}^L 2^{-i} \| \hat{g}_i \|_{L^2 (\P,d\sigma)} 
\\ &	\ll 
	\;\;\sum_{i :~ 2^{\frac{3}{2}i} \le p^{\frac{9}{4}}} 2^{-i} \left( 2^{\frac{3}{4}i} + p^{-\frac{3}{4}} 2^{\frac{3}{2}i} \right) \;\; + \;\;\sum_{i :~ p^{\frac{9}{4}}\le 2^{\frac{3}{2}i} \le p^{\frac{7}{3}}} 
	2^{-i}2^{\frac{3}{4}i}p^{\frac{3}{8}}   \\
	& +
	\sum_{i :~ p^{\frac{7}{3}} \le 2^{\frac{3}{2}i} \le p^3} 2^{-i}  2^{\frac{33}{32} i} p^{-\frac{1}{16}}
	\;\;\;\;\;\;\;\;\;\;\;+\;\;\;\;\;
	\sum_{i :~ p^3 \le 2^{\frac{3}{2}i} \le p^4} 2^{-i}  2^{\frac{3}{4}i} p^{\frac{1}{2}}
\\ &
\ll \;\;\;\;1 + 1+ 1 + 1 + 1 \\ &
\ll \;\;\;1,
\end{aligned}\label{long}\end{equation}
as required. 
\end{proof}

\subsection{Proof of Lemma \ref{l:par_energy}}
We use the point-plane incidence theorem of the first author ~\cite[Theorem 3, 3*]{Rud} (see also~\cite[Theorem~8]{MPR-NRS} and the proof of~\cite[ Corollary~2]{MuPet2}), 
which combined with the incidence bound 
from~\cite[Section~3]{MuPet1} leads to the following asymptotic formula.

\begin{lemma}
	\label{lem:Misha+}
	Let  $Q \subseteq \F^3$ be a set of points and $\Pi$ a set of planes in $F^3$. Suppose that $|Q| \le |\Pi|$ and that $k$ is the maximum number of collinear points in $Q$. Then
	$$
	\sum_{q \in Q} \sum_{\pi \in \Pi} 
	\pi (q)
	- \frac{|Q| |\Pi|}{p} \ll |Q|^{\frac{1}{2}} |\Pi| + k |Q| \,.
	$$
	Moreover, if $L$ is a set of lines in $F^3$ and one excludes incidences $(q,\pi)\in Q\times \Pi$, such that $q\in l\subset \pi$ for some $l$ in $L$, then $k$ gets replaced by the maximum number of points of $Q$ on a line not in $L$. 
\end{lemma}

\begin{proof}[Proof of Lemma \ref{l:par_energy}]

	Clearly,  if $x,y,z\in \P$, then in order that $x-z+y$ also be in $\P$, one must have $(\u x- \u z+\u y)\cdot (\u x- \u z+\u y) = \u x \cdot \u x - \u z \cdot \u z + \u y \cdot \u y$,
	which means
\begin{equation}\label{f:inc_d-1}
	(\u x - \u z) \cdot (\u z - \u y) = 0\,.
\end{equation}
Permuting the variables in the definition of energy \eqref{energy} we arrive in a simple criterion: a quadruple $(x,y,z,u)\in \mathcal P^4$ satisfies \eqref{energy} if and only if 
$(\u x,\u y, \u z,\u u)\in (F^3)^4$ is a {\em rectangle}, as opposed to generally being a parallelogram to form an additive quadruple just in $(F^3)^4$. Namely at each vertex $\u x,\u y, \u z,\u u$, the 
dot products of adjacent difference vectors is zero. 
Note that condition \eqref{f:inc_d-1} should hold at every vertex of the rectangle, hence finding all solutions of the latter relation applies to each geometric rectangle at least four times. In the sequel we will use this freedom to chose the convenient ``corner'' $(x,y,z)$ of the rectangle and the fact that the energy equals four times the number of ``geometric' rectangles.

The rectangles one encounters in $d=4$ are of three types. The first type is {\em ordinary} ones, that is lines along all the four sides are non-isotropic. These rectangles ``look like'' the Euclidean ones. The opposite case is  {\em degenerate} rectangles, namely when both sides, adjacent to a vertex, are isotropic vectors. This implies that $\u x,\u y, \u z,\u u$ all lie on the same isotropic line, for isotropic lines through a point in $F^3$ form a copy of the cone $S_0$, rather than a plane. The intermediate case comprises {\em semi-degenerate} rectangles, namely when the lines along one pair of opposite sides are isotropic, but for the other two sides -- non-isotropic. Given the line $l$ containing an isotropic side of such a rectangle (by which we mean that $l$ contains two adjacent vertices), the rectangle must lie in the unique {\em semi-isotropic plane} $l^\perp$, containing $l$. Heuristically, every parallelogram in $l^\perp$, one of whose sides is parallel to $l$ is a rectangle.

We further use $\u A$ to denote the projection of $A\subseteq \mathcal P$ on the first three variables: $A$ is a graph over $\u A$, and $|\u A|=|A|.$ We seek to count solutions of equations \eqref{f:inc_d-1}, when the variables are in $\u A$.

Obviously, a solution of equation (\ref{f:inc_d-1}) can be interpreted as a point-plane incidence in $F^3$: $\u x$ being the point and the plane $\pi$ being the one
	passing through $\u z$ and with the 
	normal vector $\u y- \u z$. In the sequel we assume that $\u y\neq \u z$, for otherwise we have $|A|^2$ trivial solutions to the energy equation.
	
	The set of planes is, in fact, a multiset. Observe that linearity of the estimate of Lemma \ref{lem:Misha+} in $|\Pi|$ enables the set of planes $\Pi$ to be a multiset as well: the estimate  of the lemma remains valid, with $m$ counting planes with multiplicities, as long as the number of distinct planes, defined by $\Pi$ is $\gg |Q|$. To this end, as to \eqref{f:inc_d-1}, we set $m=|A|^2$, $n=|A|$ and need to show that the number of distinct planes, defined by the latter relation is $\gg n$.
	
	We do this assuming $|A|\gg Cp$, with a sufficiently large absolute $C$, the number of distinct  planes, defined by pairs $(\u y,\u z)$ is $\gg |A|$, for otherwise, if $|A|\ll p$, the energy $\E(A)$ can be trivially bounded by the last term in \eqref{f:E_par_lemma}, whereupon nothing remains left to prove.
	
	Observe that given $\u y$, each value of $\u z$ yields a different plane, 
	unless the vector $\u z - \u y$ is isotropic. Hence, by the pigeonhole principle, we have $\gg|A|$ distinct planes, provided that the number of pairs $(\u y, \u z)$, such that $\u z - \u y$ is non-isotropic is $\gg|A|^2$. Let us show that this is certainly the case if $|A|\gg Cp$ for a sufficiently large $C$. 
	
	Indeed, we can build a graph $\Gamma$ on the vertex set $\u A$, connecting vertices $(\u y, \u z)$ by an edge if $\u z - \u y$ is isotropic, and we shall show that $\Gamma$ needs $\gg |A|^2$ additional edges to be turned into a complete graph.
	
	One cannot have a non-trivial triangle (that is when its  three distinct vertices are non-collinear on an isotropic line)  with vertices $\u a,\,\u b,\,\u c\,\in\, F^3$, so that each side of the triangle is isotropic. Indeed,
	$$
	\u a-\u c = (\u a-\u b) - (\u c-\u b),
	$$
	hence if both sides $\u a-\u b$  and $\u c-\u b$ are isotropic, then $\u a-\u c$ is not, for otherwise we would have $(\u a-\u b) \cdot (\u c-\u b)=0$, and hence the whole plane in $F^3$, defined by the points $\u a,\,\u b,\,\u c$ would be isotropic, which cannot happen. If $|A|\geq Cp$, for some sufficiently large absolute constant $C$, then the number of collinear point triples in $A$, lying on isotropic lines is at most $|A|p(p+1)\ll \frac{|A|^3}{C^2}$. (From each point there are at most $p+1$ isotropic directions, however we could use a cruder estimate $|A|^2p$). A collinear triple on an isotropic line is the only way to have a triangle in the graph $\Gamma$, so $\Gamma$ has a small triangle density $\ll C^{-2}$. On the other hand, if by the pigeonhole principle the edge density of $\Gamma$ were as large as $(1-c)$, for a sufficiently small constant $c>0$ -- which would mean that $\u z -\u y$ were non-isotropic for only $\ll c|A|^2$ pairs $(\u y, \u z)$ -- then the triangle density would be at least $1-O(c)$. In fact, one can use a much more fine-tuned asymptotic formula by Razborov \cite{Raz}, which tells that that the triangle density would be at least $1-3c + O(c^2)$.

	We continue under the assumption that $|A|\gg p$ and applying Lemma \ref{lem:Misha+} obtain an intermediate bound for energy 
$$
\E(A) \ll \frac{|A|^3}{p} + |A|^{\frac{5}{2}} + k|A|^2  \,,
$$
where $k$ is the maximum number of collinear points in $\u A$. Clearly, if $k\ll |A|^{\frac{1}{2}}$, we are done.

\bigskip 

Otherwise, we proceed as follows.  Let $L$ be the set of all lines in $F^3$, supporting, say $\geq 10 |A|^{\frac{1}{2}}$ points of $\u A$. By excluding the incidences along the lines in $L$ in the application of Lemma \ref{lem:Misha+}, we succeed in counting all the rectangles in $\u A$, such that at least one line containing the side of the rectangle is not in $L$. The number $\E_1$ of such rectangles is

\begin{equation}\label{f:E_par1}
\E_1 \ll \frac{|A|^3}{p} + |A|^{\frac{5}{2}} \,,
\end{equation}

 Let $L'\subseteq L$ be the subset of  non-isotropic lines and $\u A'$ the subset of $\u A$, supported on the union of these lines, and $A'$ its lift on $\mathcal P$.
We now bound $\E(A',A)$. This count will include all ordinary and semi-degenerate rectangles, the line supporting one of whose sides contains $\geq 10 |A|^{\frac{1}{2}}$ points of $\u A$. Let $\E_2\leq \E(A',A)$ be the number of such rectangles.

By the exclusion-inclusion principle, $\u A'$ is supported on some minimal set $L'$ of $|L'|\ll  |A'| |A|^{-\frac{1}{2}}$ non-isotropic lines with at least $10 |A|^{\frac{1}{2}}$ points of $\u A$ on each line. 

Let $\u A_l := l \cap \u A'$, for a line $l\in L'$, $A_l$ denoting the lift  of $\u A_l$ on $\mathcal P$.

Let us show that
\begin{equation}\label{mix}\E(A_l,A) = |\{ ( x,  y,  z, u) \in A\times A_l\times A\times A_l:\;  x+  y =  z + u\}|\ll |A_l||A|\,.\end{equation}
Geometrically speaking, if we fix a diagonal of a rectangle, whose one side lies on a non-isotropic line, this fixes the remaining two vertices. 

Analytically, without loss of generality, the line $l$ can be parametrised as $x_1=t$, $x_2=\a t+ \beta$, $x_3=\gamma t +\delta$ for some $(\a,\beta,\gamma,\delta)\in F^4$.  

Since $A \subseteq \mathcal{P}$, then a quadruple satisfying the energy equation in \eqref{mix} is described by the system of equations
$$\begin{array}{rcl}
	t^2 (1+\a^2 +\gamma^2) + 2 t(\a \beta + \gamma \d ) + \u x \cdot \u x &=& s^2 (1+\a^2 +\gamma^2) + 2s(\a \beta + \gamma \d )  + \u z \cdot \u z\,,
\\  \hfill \\
	\u x - \u z & = & (s-t) (1,\a, \gamma) \,.\end{array}
$$
The second equation is vacuous if  $s=t$ and $\u x = \u z$, yielding $|A| |A_l|$ trivial solutions to \eqref{mix}. Otherwise we eliminate $\u x$, getting $$
	t^2 (1+\a^2 +\gamma^2)  +  t(\a \beta + \gamma \d ) - ts (1+\a^2 +\gamma^2) 
		=
 	s(\a \beta + \gamma \d ) + (t-s) \u z \cdot  (1,\a, \gamma) \,,
$$
a quadratic equation in $t$. Since the line $l$ is non-isotropic, the leading coefficient
$1+\a^2 +\gamma^2 \neq 0$. 
This gives at most $2|A_l| |A|$ nontrivial solutions  to the equation in \eqref{mix}.  

It follows from \eqref{mix}  by the Cauchy-Schwarz inequality and since $|L'|\ll |A|^{\frac{1}{2}}$ that
\begin{equation}\label{f:E_par2}
	\E_2 \le \left( \sum_{l\in L'} \E(A,A_l)^{\frac{1}{2}} \right)^{2} \ll |A| \left(\sum_{l\in L'} |A_l|^{\frac{1}{2}} \right)^{2}
		\le
			|A| |L'| |A'| \; \ll \; |A|^{\frac{5}{2}}.
\end{equation}

It remains to count the degenerate rectangles on lines in $L$, whose number we denote as $\E_3$.
It's easy to see that $\E_3 \ll |A|p^2$. Indeed, take a dyadic subset of isotropic lines in $L$, each supporting $\sim k$ points of $\u A$. Then the maximum number of degenerate rectangles supported on these lines is trivially bounded as $\ll \frac{|A|}{k} k^3$. Since $k\leq p$, the dyadic sum is dominated by the term with the largest value of $k$, that is $k=p$.

Combining this bound with bounds \eqref{f:E_par1}, \eqref{f:E_par2}, since $\E(A)\ll\E_1+\E_2+\E_3$, completes the proof of Lemma \ref{l:par_energy}. 
\end{proof}

\section{Proof of Theorem \ref{t:rest3}}
We set $d=3$, $V=F^3$ and assume that the base plane $F^2$ has no isotropic directions, that is $p\equiv 3 \pmod 4$. Whether or not there are isotropic directions, interpolating just between formulae \eqref{f:g_hat_2}  and \eqref{f:g_hat_3} yields a variant of Corollary \ref{cor:g_hat_bounds}, which results in the Stein-Tomas extension estimate $L^2\to L^4$, which the presence of isotropic lines makes the best possible one if $p\equiv 1 \pmod 4$.

Since the structure of the proof of Theorem \ref{t:rest3} is similar to that of the proof of Theorem \ref{t:rest}, we omit some tedious straightforward calculations in the sequel. 

We first present an incidence  statement that will play the same role as Lemma \ref{l:par_energy} played in the case $d=4$.
Actually, the second bound of  Lemma \ref{l:corners} coincides with \cite[Lemma 4.1]{IKL} but for us it is an unconditional bound of relatively little consequence. 
\begin{lemma}\label{l:corners}
For $A\subseteq \mathcal{P}\subset V,$ one has
\begin{equation}\label{f:E2}
\E(A) \;\ll \;\left\{ \begin{array}{lll} |A|^{\frac{17}{7}} &for & |A|\leq p^{\frac{26}{21}}\,,\\ \hfill \\
 \frac{|A|^3}{p}  + |A|^2\sqrt{p}\,.\end{array}\right.
\end{equation}\end{lemma}	
The proof of Lemma \ref{l:corners} is presented in a separate section.

Substituting these bounds as $\E(G_h)$ in the estimate  \eqref{e:g^-energy} of Lemma \ref{l:g^-energy} and optimising with the Stein-Tomas inequalities \eqref{f:g_hat_2}, \eqref{f:g_hat_3} yields the following statement, an an analogue of Corollary \ref{cor:g_hat_bounds}.

\begin{corollary}
	Let $g: V \to \mathbb{C}$ be a function such that $\| g\|_\infty \leq 1$ on its support $G$. 
	Then  
\begin{displaymath}
	\| \hat{g} \|_{L^2 (\P,d\sigma)} \ll \left\{ \begin{array}{ll}
|G|^{\frac{1}{2}} + |G| p^{-\frac{1}{2}} & \textrm{for } \;1\le |G| \le p^{\frac{16}{9}}\,, \\
|G|^{\frac{19}{28}} p^{\frac{1}{14}}+|G|^{\frac{7}{8}}p^{-\frac{125}{336}} & \textrm{for } \;p^{\frac{16}{9}} \le |G| \le p^{\frac{47}{21}}\,,  \\
|G|^{\frac{5}{8}} p^{\frac{3}{16}} & \textrm{for } \;p^{\frac{47}{21}} \le |G| \le p^{\frac{5}{2}}\,,\\
|G|^{\frac{1}{2}} p^{\frac{1}{2}} & \textrm{for } \;p^{\frac{5}{2}} \le |G| \le p^3\,.
\end{array} \right.
\end{displaymath}
\label{cor:g_hat_bounds2d}
\end{corollary}

\begin{proof} We start with the last estimate, namely the bound \eqref{f:g_hat_3}, which is better than the unconditional bound in the second line of 
\eqref{f:E2} for $|G|\geq p^{\frac{5}{2}}.$

Furthermore, when $p^{\frac{47}{21}} \le |G| \le p^{\frac{5}{2}}$ we use the unconditional bound in the second line of 
\eqref{f:E2}: substituting it for each $\E(G_h)$ in formula \eqref{e:g^-energy} and applying the H\"older inequality yields an estimate
$$\begin{aligned} |G|^{\frac{3}{8}} p^{-\frac{1}{8}}\left( p^{-\frac{1}{8}} \left( \sum_{h\in F} |G_h|^{\frac{3}{4}} \right )^{\frac{1}{2}} +  p^{\frac{1}{16}} \left( \sum_{h\in F} |G_h|^{\frac{1}{2}} \right )^{\frac{1}{2}}\right) \\ \leq \;\;\;
|G|^{\frac{3}{4}}p^{-\frac{1}{8}} + |G|^{\frac{5}{8}}p^{\frac{3}{16}} \;\; \ll  \;\;|G|^{\frac{5}{8}}p^{\frac{3}{16}}\,, \end{aligned}$$
given the above range of $|G|$.

We now go back to the other end of the claim of Corollary \ref{cor:g_hat_bounds2d}, and 
for $ |G| \le p^{\frac{16}{9}}$ use the Stein-Tomas bound  \eqref{f:g_hat_2}.
At the upper end of the latter interval of $|G|$ the Stein-Tomas bound meets the following estimate, for which we use the first bound of \eqref{f:E2} for $|G_h|$ as long as  $|G_h|\leq p^{\frac{26}{21}}$ and otherwise the second bound, but the number of such $h$ is at most $p^{\frac{16}{9}-\frac{26}{21}} < p^{\frac{4}{7}}$ -- in the following calculation we index these slices of $G$ by $i$. 

Substitute these in 
 \eqref{e:g^-energy}, apply the H\"older inequality and pick the dominating term, which comes from the first bound of \eqref{f:E2} -- this yields a bound
  $$\begin{aligned} \;\;\;|G|^{\frac{3}{8}} p^{-\frac{1}{8}}\left( \left( \sum_{h\in F} |G_h|^{\frac{17}{28}} \right )^{\frac{1}{2}}  +  p^{-\frac{1}{8}}\left(\sum_{i=1,\ldots,p^{\frac{4}{7}}} |G_i|^{\frac{3}{4}} \right )^{\frac{1}{2}} +  p^{\frac{1}{16}} \left( \sum_{i=1,\ldots,p^{\frac{4}{7}}} |G_i|^{\frac{1}{2}} \right )^{\frac{1}{2}}\right) \\ \leq \;\;\;
|G|^{\frac{19}{28}} p^{\frac{1}{14}}+|G|^{\frac{3}{4}}p^{-\frac{5}{28}} + |G|^{\frac{5}{8}}p^{\frac{9}{112}} \;\; \ll  \;\;|G|^{\frac{19}{28}} p^{\frac{1}{14}}\,, \end{aligned}$$
in the above range of $|G|$. The latter estimate, naturally, meets estimate \eqref{f:g_hat_2} when $ |G| \sim p^{\frac{16}{9}}$.

Finally, for $p^{\frac{16}{9}} \le |G| \le p^{\frac{47}{21}}$ we repeat the latter calculation, only with the range of $i$, marking slices of $G$ with $|G_h|\geq p^{\frac{26}{21}}$, ranging from $1$ to $|G|p^{-\frac{26}{21}}$.  This yields the estimate
  $$\begin{aligned} |G|^{\frac{19}{28}} p^{\frac{1}{14}}+|G|^{\frac{3}{8}} p^{-\frac{1}{8}}\left(  p^{-\frac{1}{8}}\left(\sum_{i=1,\ldots,|G|p^{-\frac{26}{21}}} |G_i|^{\frac{3}{4}} \right )^{\frac{1}{2}} +  p^{\frac{1}{16}} \left( \sum_{i=1,\ldots,|G|p^{-\frac{26}{21}}} |G_i|^{\frac{1}{2}} \right )^{\frac{1}{2}}\right) \\ \leq \;\;\;
|G|^{\frac{19}{28}} p^{\frac{1}{14}} +|G|^{\frac{7}{8}}p^{-\frac{1}{4}-\frac{13}{84}} + |G|^{\frac{7}{8}}p^{-\frac{1}{16}-\frac{13}{42}}  \;\;\ll  \;\; |G|^{\frac{19}{28}} p^{\frac{1}{14}}+|G|^{\frac{7}{8}}p^{-\frac{125}{336}}\,. \end{aligned}$$

\end{proof}


\begin{proof}[Proof of Theorem \ref{t:rest3}]
The statement of the theorem now follows from the bounds of Corollary \ref{cor:g_hat_bounds2d}, similar to \eqref{long} in the proof of Theorem \ref{t:rest}. 

 We only present a part of the calculation, which yields the worst extension exponent. This is contributed, naturally, by the second range in Corollary \ref{cor:g_hat_bounds2d}, which has come from the Stevens-de Zeeuw Theorem. 

Assume that for some $r \in (1,3/2)$ one has
$$
	\sum_{x\in V} |g(x)|^{r} = 1 \,.
$$
We need to ensure, for the corresponding group of dyadic pieces $g_i$  of $g$ (with  $g(x)\sim 2^{-i}$ on a dyadic piece) that the following estimate holds:

$$
	1\gg \sum_{i :~ p^{\frac{16}{9}} \le 2^{ri} } 2^{-i}  \left( 2^{\frac{19}{28} ri} p^{\frac{1}{14}}+2^{\frac{7}{8} ri}p^{-\frac{125}{336}}\right).
$$
This means $r$ should satisfy $\frac{1}{r} = \frac{19}{28}+ \frac{1}{14}\cdot\frac{9}{16}$ (the second term in the above estimate can be easily seen to play no role), whence $r=\frac{224}{161}.$ Therefore, the dual Fourier extension exponent is $\frac{32}{9}.$ 

This concludes the proof of Theorem \ref{t:rest3}. \end{proof}

\subsection{Proof of Lemma \ref{l:corners}}
We start off with a geometric incidence bound to replace Lemma \ref{lem:Misha+}.
\begin{lemma}
	\label{lem:2d}
	The number of incidences $I(n)$ between $m$ lines and $n$ points in $F^2$ is bounded as \begin{equation}\label{ins} I(n) \ll \left\{
	\begin{array}{lll} (mn)^{\frac{11}{15}} + m+n & for\;\;\;n^{13}m^{-2}\leq p^{15}\,, \\ \hfill \\
	 \frac{mn}{p} +  (mnp)^{\frac{1}{2}}\,.\end{array}\right. \end{equation} \end{lemma}
Lemma \ref{lem:2d} is an amalgamation of incidence theorems of Stevens and de Zeeuw \cite[Theorem 3]{SdZ} and a well-known theorem of Vinh \cite[Theorem 3]{Vinh}. The former theorem, which contributes the first line in estimate \eqref{ins} is valid for small sets. The latter theorem, proved just by linear algebra, is nontrivial  if both $m,n>p.$

\begin{proof}[Proof of Lemma \ref{l:corners}]We restate equation \eqref{f:inc_d-1} in the proof of Lemma \ref{l:par_energy}, aiming to bound the number of solutions of the equation 
\begin{equation}
(\u x- \u z)\cdot(\u z- \u y)=0:\;\;\u x,\u y, \u z\in \u A\,,
\label{right}\end{equation} where $\u A$ is the projection of $A$ on the $(x_1,x_2)-$plane.  Let us set $|A|=|\u A|=n$.

Equation \eqref{right} is a well-known problem of counting the maximum number of right triangles with vertices in the plane point set $\u A$, which in the real case was given a sharp answer by Pach and Sharir \cite{PS}, by using the Szemer\'edi-Trotter theorem. Here we adapt the argument in order to use the Stevens-de Zeeuw incidence bound -- the first estimate in \eqref{ins} -- instead.

Note that the second bound in \eqref{ins} always provides a universal bound
\begin{equation}\label{uni}
\E(A)\ll \frac{|A|^3}{p} + |A|^2\sqrt{p}.
\end{equation}

We next recast the Stevens-de Zeeuw incidence bound in the usual way, aiming at the cardinality $m_k$ of the set of $k$-rich lines, that is lines, supporting $\geq k$ points of a $n$-point set:
$$
m_k  \ll  \frac{n^{\frac{11}{4}}} {k^{\frac{15}{4}}} + \frac{n}{k} + \frac{n^{\frac{13}{2}}}{p^{\frac{15}{2}}}\,,
$$
The third term in the bound arises as the alternative to the constraint $n^{13}m_k^{-2}\leq p^{15}$ of Lemma \ref{lem:2d} when the Stevens-de Zeeuw theorem may not apply.

Let us for technical purposes loosen this bound by subsuming the last term in the increased second term (clearly, $k\leq p$) as follows:

\begin{equation}\label{k-rich}
m_k\ll  \frac{n^{\frac{11}{4}}} {k^{\frac{15}{4}}} + \frac{n^{\frac{5}{4}}}{k} \qquad for \;\;\; n\leq p^{\frac{26}{21}}\,.
	\end{equation}
This apparently crude step is of little consequence but eases some calculations in the sequel: its positive effect is making the range of $n$ when the next key estimate \eqref{thin} is applicable somewhat wider, while the negative effect is that following estimate \eqref{e:vr} ends up being worse than it may be, yet still better than the desired  \eqref{thin}.

Next we are going to show that the number of nontrivial solutions (that is with $\u x\neq \u z$ and $\u y\neq \u z$) $N$ of equation \eqref{right} satisfies the following bound:
\begin{equation}\label{thin}
for \;\;\; n\leq p^{\frac{26}{21}}, \qquad N\ll n^{2+\frac{3}{7}}.
\end{equation}

Assume $n\leq p^{\frac{26}{21}}$, let us express the quantity $N$  as follows. For $\u z\in A$, define $L_z$ as the set of all the $p+1$ lines in $F^2$ incident to $z$. For any line $l$ in $F^2$, let $n(l)$ be the number of points of $\u A$, supported on $l$ {\em minus 1}. Then
\begin{equation}\label{f:N}
N = \sum_{\u z \in \u A}\, \sum_{l\in L_z} n(l) n(l^\perp),
\end{equation}
where $l^\perp$ is the line perpendicular to $l$.

Let us set up a cut-off value $k_*= n^{\frac{3}{7}}$ of $n(l)$ to be justified. Partition, for every $z$, the lines $l\in L_z$ to poor ones, that is those with $n(l) \leq k_*$, and otherwise rich. 
Accordingly partition $N=N_{poor} + N_{rich}$, where  the term $N_{poor}$ means that at least one of $l,l^\perp$ under summation in \eqref{f:N} is poor, hence the alternative $N_{rich}$ is when both $l,l^\perp$ are rich. Clearly
$$
N_{poor} \leq 2k_* n^2.
$$
Let us now bound $N_{rich}$. Observe that the two terms in estimate \eqref{k-rich} meet when 
 $k\sim n^{\frac{6}{11}}$. Let us call the lines with $n(l) \geq n^{\frac{6}{11}}$ {\em very rich} and partition
 $$N_{rich} = N_{very-rich}+N_{just-rich}\,,$$
 the first term corresponding, for each $\u z$, to the sub-sum, corresponding to the case when one of $l,l^\perp$ is very rich.  Then one can bound $N_{very-rich}$ trivially, using the second term in \eqref{k-rich} and dyadic summation in  $k\geq n^{\frac{6}{11}}$  as
\begin{equation}\label{e:vr}
N_{very-rich}\ll n \sum_{l:\,n(l)\geq n^{\frac{6}{11}}} n(l) \ll n^\frac{9}{4}\log n,
\end{equation}
 which is better than 
 \eqref{thin}. 
 Indeed, given a very-rich line $l$ we count all triangles with vertices $\u x, \u z,\u y$, such that $\u z\in l$ and $\u y$ is any point outside $l$; the two will determine the third vertex $\u x\in l$.  

What is left to consider is the case of the summation in \eqref{f:N} when both  $n^{\frac{3}{7}}\leq n(l),\,n(l^\perp)\leq n^{\frac{6}{11}}$. We apply Cauchy-Schwarz to obtain 
$$
N_{just-rich}  \leq \sum_{\u z \in \u A} \;\;\sum_{l\in L_z:\,p^{\frac{3}{7}}\leq n(l) \leq n^{\frac{6}{11} } } n^2(l) .$$
The expression in the right-hand side counts collinear triples of points in $\u A$ on rich, but not very rich lines. The number of such lines with $n(l)\sim k$, for the range of $n(l)$ in question is bounded by the first term in estimate \eqref{k-rich}.  Multiplication of the latter term by $k^3$ followed by dyadic summation in $k$ yields
$$
N_{just-rich} \ll n^{\frac{11}{4}}  k_*^{-\frac{3}{4}}\,,$$ optimising with $N_{poor}\leq n^2k_*$ justifies the choice of $k_*=n^{\frac{3}{7}}$, and proves \eqref{thin}. Together with the better bound \eqref{e:vr} this completes the proof of Lemma \ref{l:corners}.
\end{proof}

\subsection*{Final remark} Mockenhaupt and Tao conjectured that in $d=3$, the sharp extension exponent from $L^2(\mathcal P,d\sigma)$ should be $3$. However, the strategy in this paper (as well as \cite{MoT}, \cite{Lew}, etc.) can only yield an extension exponent $>\frac{10}{3}$ (which one would get with the full strength Szemer\'edi-Trotter theorem). Indeed, as we have seen, the energy $\E(G_h)$ under summation in the right-hand side of \eqref{e:g^-energy} equals the number of rectangles on the corresponding horizontal slice of the support $G$ of a bounded function $g$ on $F^3$. Choosing $g\equiv 1$ on its support and each horizontal slice of $G$  as a translate or dilate of the Cartesian product of the interval $[1,\ldots,N]$ with itself, the number of rectangles on each slice is $\gg N^4\log N$. An easy calculation with, say, choosing $N\sim p^{\frac{1}{3}}$ and having $p$ such slices of $G$ shows that the minimum of the estimates \eqref{f:g_hat_2}, \eqref{f:g_hat_3} and \eqref{e:g^-energy} is always greater than $p^{\frac{7}{6}}$ times a power of $\log p$. This means that the restriction exponent is $<\frac{10}{7}$, hence the extension exponent from $L^2(\mathcal P,d\sigma)$  is $>\frac{10}{3}.$



\end{document}